\documentclass[reqno, 11pt]{amsart}
\usepackage{amsthm, amsmath,amssymb,cite}

\usepackage{bbm,bm}

\usepackage[margin=1in]{geometry}

\usepackage[textsize=small,backgroundcolor=orange!20]{todonotes}

\usepackage[pdfstartview=FitV, 
pagebackref=false]{hyperref}

\usepackage{enumerate}

\usepackage{algorithm,algorithmic}

\usepackage[color,final]{showkeys} 

\colorlet{refkey}{orange!20}
\colorlet{labelkey}{blue!30}

\newtheorem{theorem}{Theorem}[section]
\newtheorem*{theorem*}{Theorem}

\newtheorem{lemma}[theorem]{Lemma}

\newtheorem*{question*}{Question}

\theoremstyle{definition}
\newtheorem{definition}[theorem]{Definition}
\newtheorem*{definition*}{Definition}

\theoremstyle{remark}
\newtheorem*{remark}{Remark}

\newcommand{\abs}[1]{\left\lvert#1\right\rvert}

\newcommand{\tr}{\operatorname{tr}}

\newcommand{\RR}{\mathbb{R}}

\newcommand{\cP}{\mathcal{P}}

\newcommand{\dd}{d}

\newcommand{\R}{\mathbb{R}}
\newcommand{\vu}{\mathbf{u}}
\newcommand{\vv}{\mathbf{v}}

\newcommand{\va}{\mathbf{a}}
\newcommand{\iskt}{\substack{i \in S\\k \in T}}

\author{Jacob Fox}
\address{Department of Mathematics, Stanford University, Stanford, CA 94305.}
\email{jacobfox@stanford.edu}
\thanks{J.~Fox is supported by a Packard Fellowship, by NSF CAREER award DMS 1352121,
and by an Alfred P. Sloan Fellowship}

\author{L\'aszl\'o Mikl\'os Lov\'asz}
\address{Department of Mathematics\\ UCLA\\ Los Angeles, CA 90095.}
\email{lmlovasz@math.ucla.edu}
\thanks{L.~M.~Lov\'asz is supported by NSF Postdoctoral Fellowship Award DMS 1705204.} 

\author{Yufei Zhao}
\address{Department of Mathematics\\ MIT\\ Cambridge, MA 02139.}
\email{yufeiz@mit.edu}
\thanks{Y.~Zhao is partially supported by NSF award DMS 1362326.}

\title{A fast new algorithm for weak graph regularity}

\date{\today}

\begin{document}

\begin{abstract}
	We provide a deterministic algorithm that finds, in $\epsilon^{-O(1)} n^2$ time, an $\epsilon$-regular Frieze--Kannan partition of a graph on $n$ vertices. The algorithm outputs an approximation of a given graph as a weighted sum of $\epsilon^{-O(1)}$ many complete bipartite graphs.
	
	As a corollary, we give a deterministic algorithm for estimating the number of copies of $H$ in an $n$-vertex graph $G$ up to an additive error of at most $\epsilon n^{v(H)}$, in time $\epsilon^{-O_H(1)}n^2$.
\end{abstract}

\subjclass[2010]{05C85, 05C50, 05D99}

\maketitle

\section{Introduction}

The regularity method, based on Szemer\'edi's regularity lemma \cite{Sz76}, is one of the most powerful tools in graph theory. Szemer\'edi \cite{Sz75} used an early version in the proof of his celebrated theorem on long arithmetic progressions in dense subsets of the integers. Roughly speaking, the regularity lemma says that every large graph can be partitioned into a small number of parts such that the bipartite subgraph between almost every pair of parts is random-like. 
One of the main drawbacks of the original regularity lemma is that it requires a tower-type number of parts, where the height of the tower depends on an error parameter $\epsilon$. However, for many applications, the full power of the regularity lemma is not needed, and a weaker notion of Frieze-Kannan regularity suffices.

To state the regularity lemmas requires some terminology. Let $G$ be a graph, and $X$ and $Y$ be (not necessarily disjoint) vertex subsets. Let $e(X,Y)$ denote the number of pairs vertices $(x,y) \in X \times Y$ that are edges of $G$. The {\it edge density} $d(X,Y)= e(X,Y) / (|X||Y|)$ between $X$ and $Y$ is the fraction of pairs in $X \times Y$ that are edges. The pair $(X,Y)$ is {\it $\epsilon$-regular} if for all $X' \subseteq X$ and $Y' \subseteq Y$ with $|X'| \geq \epsilon|X|$ and $|Y'| \geq \epsilon |Y|$, we have $|d(X',Y')-d(X,Y)|<\epsilon$. Qualitatively, a pair of parts is $\epsilon$-regular with small $\epsilon$ if the edge densities between pairs of large subsets are all roughly the same. A vertex partition $V=V_1 \cup \ldots \cup V_k$ is {\it equitable} if the parts have size as equal as possible, that is we have $||V_i|-|V_j|| \leq 1$ for all $i,j$. An equitable vertex partition with $k$ parts is {\it $\epsilon$-regular} if all but $\epsilon k^2$ pairs of parts $(V_i,V_j)$ are $\epsilon$-regular. The regularity lemma states that for every $\epsilon>0$ there is a (least) integer $K(\epsilon)$ such that every graph has an $\epsilon$-regular equitable vertex partition into at most $K(\epsilon)$ parts. 


To state Frieze-Kannan regularity precisely, first, we extend the definition of $e(X,Y)$ and $d(X,Y)$ to weighted graphs. Below by \emph{weighted graph} we mean a graph with edge-weights. Given two sets of vertices $X$ and $Y$, we let $e(X,Y)$ denote the sum of the edge-weights over pairs $(x,y) \in X \times Y$ (taking $0$ if a pair does not have an edge). Let $d(X,Y)=e(X,Y)/(|X||Y|)$ as earlier. Recall that the cut metric $\dd_\square$ between two graphs $G$ and $H$ on the same vertex set $V = V(G) = V(H)$ is defined by
\[
\dd_\square(G,H) := \max_{U,W \subseteq V} \frac{|e_G(U,W) - e_H(U,W)|}{|V|^2},
\] and this extends to graphs with weighted edges, and can be adapted to bipartite graphs (with given bipartitions). Given any edge-weighted graph $G$ and any partition $\cP \colon V = V_1 \cup V_2 \cup \dots \cup V_t$ of the vertex set of $G$ into $t$ parts, let $G_\cP$ denote the weighted graph with vertex set $V$ obtained by giving weight $d_{ij} := d(V_i,V_j)$ to all pairs of vertices in $V_i \times V_j$, for every $1 \le i \le j \le t$. We say $\cP$ is an \emph{$\epsilon$-regular Frieze--Kannan} (or \emph{$\epsilon$-FK-regular}) partition if $\dd_\square(G,G_\cP)\le \epsilon$. In other words, $\cP$ is an $\epsilon$-regular Frieze--Kannan partition if
\begin{equation} \label{eq:FK-pair}
\left\lvert e(S,T) - \sum_{i,j=1}^t d_{ij} |S \cap V_i||T \cap V_j| \right\rvert \le \epsilon |V|^2.
\end{equation}
for all $S,T \subseteq V$. We say that sets $S$ and $T$ \emph{witness} that $\cP$ is not $\epsilon$-FK-regular if the above inequality is violated.

Frieze and Kannan~\cite{FK96,FK99a} proved the following regularity lemma. 

\begin{theorem}[Frieze--Kannan]
	Let $\epsilon > 0$. Every graph has an $\epsilon$-regular Frieze--Kannan partition with at most $2^{2/\epsilon^2}$ parts.
\end{theorem}

There is a variant of the weak regularity lemma, where the final output is not a partition of $V$ into $2^{\epsilon^{-O(1)}}$ parts, but rather an approximation of the graphs as a sum of $\epsilon^{-O(1)}$ complete bipartite graphs, each assigned some (not necessarily nonnegative) weight, see \cite{FK99a}. For $S,T \subseteq V$, we denote by $K_{S,T}$ the weighted graph where an edge $\{s,t\}$ has weight 1 if $s \in S$ and $t \in T$ (and weight 2 if $s,t \in S \cap T$) and weight zero otherwise. For any $c \in \RR$, by $c G$ we mean the weighted graph obtained from $G$ by multiplying every edge-weight by $c$. For a pair of weighted graphs $G_1,G_2$ on the same set of vertices, we will use the notation $G_1+G_2$ to denote the graph on the same vertex set with edge weights summed (and weight $0$ corresponding to not having an edge). Additionally, we write $c$ to mean the constant graph with all edge-weights equal to $c$. We also use $d(G):=d(V(G),V(G))$ to denote the edge density of the weighted graph $G$. 

\begin{theorem}[Frieze--Kannan] \label{thm:FK-overlay}
	Let $\epsilon > 0$. Let $G$ be any weighted graph with $[-1,1]$-valued edge weights. There exists an $r = O(\epsilon^{-2})$, and there exist subsets $S_1, \dots, S_r, T_1, \dots, T_r \subseteq V$, and $c_1, \dots, c_k \in [-1,1]$, so that  
	\[
	\dd_\square(G, d(G) + c_1K_{S_1, T_1} + \dots + c_r K_{S_r,T_r}) \le \epsilon.
	\]
\end{theorem}

See \cite[Lemma 4.1]{LS07} for a simple proof (given there in a more general setting of arbitrary Hilbert spaces). It is well known using the triangle inequality (see, e.g., \cite{FK99a}) that given sets and numbers as in the theorem, the common refinement of all $S_i,T_i$ must be a $2\epsilon$-regular Frieze-Kannan partition.

In addition to proving that a partition or ``cut graph decomposition'' exists, Frieze and Kannan gave probabilistic algorithms for finding a  weak regular partition \cite{FK96,FK99a} or decomposition. Two deterministic algorithms were given by Dellamonica, Kalyanasundaram, Martin, R\"odl, and Shapira \cite{DKMRS12,DKMRS15}. Specifically, in \cite{DKMRS12}, the authors gave an $\epsilon^{-6} n^{\omega+o(1)}$ time algorithm ($\omega < 2.373$ is the matrix multiplication exponent) to generate an equitable $\epsilon$-regular Frieze--Kannan partition of a graph on $n$ vertices into at most $2^{O(\epsilon^{-7})}$ parts. In \cite{DKMRS15} a different algorithm was given which improved the dependence of the running time on $n$ from $O_\epsilon(n^{\omega+o(1)})$ to $O_\epsilon(n^2)$, while sacrificing the dependence of $\epsilon$. Namely, it was shown that there is a deterministic algorithm that finds, in $2^{2^{\epsilon^{-O(1)}}} n^2$ time, an $\epsilon$-regular Frieze--Kannan partition into at most $2^{\epsilon^{-O(1)}}$ parts.

In Section \ref{secalgreg}, we give an optimal algorithm that provides the best of both worlds: We give an algorithm that finds, in $\epsilon^{-O(1)}n^2$ time, a weakly regular partition.\footnote{Theorem~\ref{thmalgweakregbest} replaces \cite[Corollary~3.5]{FLZ17}, which we retracted~\cite{FLZ17err}.} In fact, we provide an algorithm for finding a cut graph decomposition, which is more useful in some applications. The algorithm is also self-contained. 

 
\begin{theorem} \label{thmalgweakregbest}
	There is a deterministic algorithm that, given $\epsilon>0$ and an $n$-vertex graph $G$, outputs, in $\epsilon^{-O(1)}n^2$ time, subsets $S_1,S_2,...,S_r,T_1,T_2,...,T_r \subseteq V(G)$ and $c_1,c_2,...,c_r \in \{-\frac{\epsilon^8}{300},\frac{\epsilon^8}{300}\}$ for some $r = O(\epsilon^{-16})$, such that
	\[
	\dd_\square(G, d(G) + c_1K_{S_1, T_1} + \dots + c_r K_{S_r,T_r}) \le \epsilon.
	\] 
\end{theorem}

\begin{remark}
	Given a decomposition as above, we obtain the $2\epsilon$-regular Frieze-Kannan partition that gives the common refinement of all $S_i,T_i$ in time $O(nr)$, by going through the vertices of the graph, and checking, for each vertex, which parts it does and does not belong to.
\end{remark}

\begin{remark}
As in the case of the usual regularity lemma, it is possible to obtain an equitable partition in the Frieze--Kannan regularity lemma, increasing the number of parts and the cut distance by a negligible amount. This can be done by arbitrarily partitioning each part into essentially equal size parts of the desired size and a remainder part, and then arbitrarily partitioning the union of the remainder vertices into parts of the desired size. We leave the details of this algorithm to the reader.
\end{remark}

In Section \ref{sec:counts}, using the above algorithmic weak regularity lemma, we obtain a deterministic algorithm for approximating the number of copies of a fixed vertex graph $H$ in a large vertex graph $G$.\footnote{Theorem~\ref{thm:countingalgthmgood} replaces \cite[Theorem 1.4]{FLZ17}, which we retracted \cite{FLZ17err}.} Note that there is an easy randomized algorithm for estimating the number of copies of $H$ by sampling. However, it appears to be nontrivial to estimate this quantity deterministically. Duke, Lefmann and R\"odl~\cite{DLR} gave an approximation algorithm for the number of copies of a $k$-vertex graph $H$ in an $n$-vertex graph $G$ up to an error of at most $\epsilon n^k$ in time $O(2^{(k/\epsilon)^{O(1)}} n^{\omega+o(1)})$. We give a new algorithm which significantly improves the running time dependence on both $n$ and $\epsilon$.

\begin{theorem}\label{thm:countingalgthmgood}
There is a deterministic algorithm that, given $\epsilon > 0$, a graph $H$, and an $n$-vertex graph $G$, outputs, in $O(\epsilon^{-O_H(1)} n^2)$ time, the number of copies of $H$ in $G$ up to an additive error of at most $\epsilon n^{v(H)}$.
\end{theorem}

\begin{remark}
An examination of the proof shows that the exponent of $\epsilon^{-1}$ in the running time can be $9^{\abs{H}}$ (though not believed to be optimal). For example, we can count the number of cliques of order $1000$ in an $n$-vertex graph up to an additive error $n^{1000-10^{-1000000}}$ in time $O(n^{2.1})$.
\end{remark}

\begin{remark}
All results here can be generalized easily to weighted graphs $G$ with bounded edge-weights.
\end{remark}

\section{Algorithmic weak regularity} \label{secalgreg}

Here we prove Theorem \ref{thmalgweakregbest}. We will prove the following, roughly equivalent form. In order to state it, we first give some notation.
Given a matrix $A$, we denote by $\|A\|$ the spectral norm, i.e. the largest singular value.
It is well known that this is equal to the operator norm of $A$ when viewed as an operator between $L^2$-spaces.
We also use the Frobenius norm 
\[\|A\|_F=\sqrt{\sum_{i,j} a_{i,j}^2}.
\] 
and the entry-wise maximum norm
\[\|A\|_{\max}=\sup_{i,j} |a_{i,j}|.
\]
Given a set $S \subseteq [n]$, we will denote by $\bm{1}_S \in \RR^n$ the characteristic vector of $S$.

\begin{theorem} \label{thmcomplalg}
There is an algorithm that, given an $\epsilon > 0$ and a matrix $A \in [-1,1]^{n \times n}$, outputs, in $\epsilon^{-O(1)}n^2$ time, subsets $S_1, \dots, S_r, T_1, \dots, T_r \subseteq [n]$ and real numbers $c_1, \dots, c_r \in \{- \frac{\epsilon^8}{300},\frac{\epsilon^8}{300}\}$ for some $r = O(\epsilon^{-16})$, such that, setting
\[A'=\sum_{i=1}^r c_i \bm{1}_{S_i} \bm{1}_{T_i}^\top ,
\] 
each row and column of $A-A'$ has $L^2$-norm at most $\sqrt n$ (i.e. the sum of the squares of the entries is at most $n$), and
\[
\|A-A'\| \le \epsilon n.
\] 
\end{theorem}

It is well-known that if $G$ and $H$ are weighted bipartite graphs between two sets $X,Y$ of size $n$, and $A_G,A_H$ are the adjacency matrices, with rows corresponding to $X$ and columns corresponding to $Y$, then 
\[d_\square(G,H) \le \frac{\|A_G-A_H\|}{n}
.\]
Indeed, for any $S,T \subseteq [n]$, taking the characteristic vectors $\bm{1}_S$ and $\bm{1}_T$, we have
\[|e_G(S,T)-e_H(S,T)|=\left|\bm{1}_S^\top(A_G-A_H)\bm{1}_T^\top\right| \le \|A_G-A_H\|\|\|\bm{1}_S\|_2\|\bm{1}_T\|_2 \le \|A_G-A_H\|n
.\]

Therefore, this theorem indeed implies Theorem \ref{thmalgweakregbest} (taking $A$ to be $A_G-d(G)\mathbf{1}\mathbf{1}^\top$).

The proof of the Frieze--Kannan regularity lemma and its algorithmic versions, roughly speaking, run as follows:
\begin{itemize}
	\item Given a partition (starting with the trivial partition with one part), either it is $\epsilon$-FK-regular (in which case we are done), or we can exhibit some pair of subsets $S,T$ of vertices that witness the irregularity (in the algorithmic versions, one may only be guaranteed to find $S$ and $T$ that witness irregularity for some smaller value of $\epsilon$).
	\item Refine the partition by using $S$ and $T$ to split each part into at most four parts, thereby increasing the total number of parts by a factor of at most $4$. 
	\item Repeat. Use a mean square density increment argument to upper bound the number of possible iterations.
\end{itemize}

This can be modified to prove the approximation version. Roughly speaking, to find the appropriate $S_i, T_i, c_i$, in the second step of the above outline of the proof of the weak regularity lemma, instead of using $S$ and $T$ to refine the existing partition, we subtract $c \bm{1}_S\bm{1}_T^\top$ from the remaining matrix, for a carefully chosen $c$. We record the corresponding $S_i, T_i, c_i$ in step $i$ of this iteration. We can bound the number of iterations by observing that the $L^2$ norm of $A - c_1\bm{1}_{S_1}\bm{1}_{T_1}^\top - \dots - c_i\bm{1}_{S_i}\bm{1}_{T_i}^\top$ must decrease by a certain amount at each step. 

As for the algorithmic versions, the main challenge is checking whether a partition is regular, or a cut graph approximation is close in cut distance. Given a matrix $A$, up to a polynomial change in $\epsilon$, having small singular values as a fraction of $n$ is equivalent to $\tr A A^\top A A^\top$ being small as a fraction of $n^4$, which roughly says that most scalar products of rows are small as a fraction of $n$. In \cite{ADLRY}, the authors use this fact to obtain an algorithm which runs in $O(n^{\omega + o(1)})$ time and either correctly states that a pair of parts is $\epsilon$-regular, or gives a pair of subsets which realizes it is not $\epsilon^{O(1)}$-regular. This was adapted in \cite{DKMRS12} to the weak regular setting. In \cite{KRT}, the authors noticed that it suffices to check the scalar products along the edges of a well-chosen expander, which has a linear number of edges in terms of $n$, allowing them to obtain an $O_\epsilon(n^2)$-time algorithm. This was also the main idea in \cite{DKMRS15}, but their algorithm is double exponential in $\epsilon^{-1}$. A further challenge in proving Theorem \ref{thmcomplalg} with the cut matrix approximation is that the entries of the approximation matrices may not stay bounded, which was used in the algorithms for checking regularity. This is problematic, because for a general matrix $A$, the singular value (divided by $n$) and the cut norm may be quite different. To counter this, we give an algorithm which checks regularity effectively under a weaker assumption that simply the $L^2$-norm of each row and each column stays bounded. Heuristically, the reason this property is useful is that it implies that if we have a singular vector (with norm $1$) with a relatively large singular value, then no entry can be ``too large'', it must be ``spread out'', which can then be used to show that a large singular value implies a large cut norm. We then show that if we are careful, we can make sure that this property holds throughout the process. 

Let us state this more precisely. Given a matrix $A$, 
let $\va_i$ be the $i$-th row of $A$ and $\va^j$ the $j$-th column. Our main ingredient then is the following theorem. Note that in the algorithm below, the parameter $C$ affects the running time but not the discrepancy of the output sets $S,T$. 

\begin{theorem} \label{thmstep}
There exists a $(C/\epsilon)^{O(1)}n^2$ algorithm which, given a matrix $A \in \R^{n \times n}$ such that $\|A\|_{\max} \le C$, and each $\|\va_i\|_2^2 \le n$, $\|\va^j\|_2^2 \le n$ (or equivalently $\|A^\top A\|_{\max}, \|AA^\top\|_{\max} \le n$), either
\begin{itemize}
\item Correctly outputs that each singular value of $A$ is at most $\epsilon n$, or
\item Outputs sets $S,T \subseteq [n]$ such that 
\[\left|\sum_{i \in S,k \in T}a_{i,k} \right| \ge \frac{\epsilon^8}{100}n^2
.\] (This implies that $A$ has a singular value that is at least $\frac{\epsilon^8}{100}n$.)
\end{itemize}
\end{theorem}

In the next lemma, we construct the expander along which we will check the scalar products. For an integer $n$, let $J_n$ denote the $n \times n$ matrix with each entry equal to $1$.

\begin{lemma} \label{lemexpanderconstruction}
There exist fixed absolute constants $l>0$ and $0<c<1$ such that there is an algorithm which given $d_0$ and $n$, outputs a matrix $M$ on $\RR^{n \times n}$ with nonnegative integer entries, and an integer $d$ with $d_0 \le d \le ld_0$, such that
\[\|\frac{d}{n}J_n-M\| \le d^{1-c}
.\] In other words, for any vector $\vv=(v_i)_{i=1}^n \in \R^n$, we have
\begin{equation} \label{eqexpandersumsquare}
\left|\left(\sum_i v_i \right)^2-\frac{n}{d}\vv^\top M\vv\right| \le \frac{n}{d^c} \|\vv\|_2^2
.\end{equation}
The running time of the algorithm is $O(d n (\log n)^{O(1)})$.
\end{lemma}
\begin{proof}
Construct an $l$-regular two-sided expander $G_0$ on $[\widetilde{n}]$ for some $n\le \widetilde{n} \le Kn$ with $K$ fixed. This can be done in $n(\log n)^{O(1)}$ time. For example, Margulis \cite{Mar73} constructed an $8$-regular expander on $\mathbb{Z}_m \times \mathbb{Z}_m$ for every $m$, and Gabber and Galil \cite{GaGa81} showed that all other eigenvalues (besides $8$ with multiplicity $1$) are at most $5 \sqrt{2}<8$. For every vertex $(x,y) \in \mathbb{Z}_m \times \mathbb{Z}_m$, its eight neighbors are
\[(x \pm 2y,y),(x \pm (2y+1),y),(x,y \pm 2x),(x,y \pm (2x+1))
.\] Therefore we can compute, for each vertex, a list of neighbors in time $O(\log m)=O(\log n)$, which then takes $O(n \log n)$ time total. Alternatively, we can start with a Ramanujan graph for some fixed degree, constructed explicitly by Lubotzky, Phillips, and Sarnak \cite{LPS88}; Margulis \cite{Mar88}; and Morgenstern \cite{Mor94}.

The adjacency matrix $A_{G_0}$ has $A_{G_0}\mathbf{1}=l\mathbf{1}$ and all eigenvalues besides $l$ have absolute value at most some explicit $a<l$. Let $k$ be the integer and $\widetilde M=A_{G_0}^k$ be such that $d_0 \frac{n}{\widetilde{n}} \le \widetilde{d}:=l^k < ld_0 \frac{n}{\widetilde{n}}$. Note that $\widetilde M$ is symmetric and has nonnegative integer entries, so it is the adjacency matrix of some graph $G$ (possibly with multiple edges and loops). Clearly $\widetilde M\mathbf{1}=\widetilde d\mathbf{1}$,
 so $\widetilde d$ is an eigenvalue of $\widetilde M$, and all other eigenvalues have absolute value at most $a^k=a^{\log_l(\widetilde d)}=\widetilde d^{\log_l(a)}$. Since $a<l$, $\widetilde c:=1 - \log_l(a)>0$. This implies that
\[\|\frac{\widetilde d}{\widetilde{n}} J_{\widetilde{n}} - \widetilde M \| \le \widetilde d ^{1-\widetilde c}
.\]

Take any set of $n$ vertices, let $M$ be the restricted submatrix of $\widetilde M$, and let $d=\frac{n}{\widetilde{n}}\widetilde{d}$. As  $\frac{\widetilde{n}}{n} \le K$, and the spectral norm of a matrix cannot increase when taking a submatrix, we have that 

\[\|\frac{d}{n} J_{n} - M \| \le \|\frac{\widetilde d}{\widetilde{n}} J_{\widetilde{n}} - \widetilde M \| \le \widetilde d ^{1-\widetilde c}=\left(\frac{\widetilde{n}}{n}d\right)^{1-\widetilde{c}} \le \left(K d\right)^{1-\widetilde{c}} \le d^{1-c}
\]
for an explicit $c>0$.

We can construct $G_0$ in time $(\log n) ^{O(1)} n$. We make sure, for each vertex, to keep a list of its neighbors. We then compute $A_{G_0}^i$ for $i=1,2,...,k$. In each case, we make sure to keep a list of the $l^i$ neighbors of each vertex (with multiplicities). We can then compute $A_{G_0}^{i+1}$ in $O(l^i n)$ time by computing the list of $l^{i+1}$ neighbors for each vertex, by looking at its $l$  neighbors in $G_0$ and taking the (multiset) union. The total running time is therefore $O((\log n)^{O(1)}n+\sum_{i=1}^k l^i n)=O(((\log n)^{O(1)}+d)n)$.
\end{proof}
Alternatively, we could have used the zig-zag construction of expanders due to Reingold, Vadhan, and Wigderson \cite{RVW02}.

\begin{proof}[Proof of Theorem \ref{thmstep}]
Throughout this proof, we use the convention that $i$ and $j$ refer to rows, $k$ and $l$ refer to columns. The basic idea of the algorithm is the following. It is easy to see that 
\begin{equation} \label{tracefourthpower}
\tr(AA^\top AA^\top)=\sum_{i,j,k,l}  a_{i,k}a_{i,l}a_{j,k}a_{j,l}
.\end{equation}
In order to estimate this sum, we can use the expander to only compute the sum for pairs $(i,j)$ which form an edge of the expander (and then multiply by $n/d$). In fact, this is true even for the terms in (\ref{tracefourthpower}) corresponding to a fixed $k,l$. We can therefore use the expander to estimate the sum in (\ref{tracefourthpower}), and if it is large, find a $k$ for which the sum of the terms corresponding to $k$ are large. This will allow us to find sets $S,T$ as required.

Here is the algorithm.
\begin{enumerate}[1.]
\item Construct the matrix $M$ according to Lemma \ref{lemexpanderconstruction} that satisfies (\ref{eqexpandersumsquare}) (inputting $d_0=(3C^2\epsilon^{-4})^{1/c}$). Let $M=(m_{i,j})_{i,j=1}^n$. \label{stepconstraman}
\item For each $i,j$ with $m_{i,j}>0$, compute $s_{i,j}=\langle \va_i,\va_j \rangle$. \label{stepcomputems}
\item For each $k \in [n]$, compute 
\[b_k = \sum_{i,j=1}^n m_{i,j}a_{i,k}a_{j,k} s_{i,j} 
.\] \label{stepcomputebs}
\item If each $b_k \le \frac{2}{3}\epsilon^4 dn^2$, return that $\|A\| \le \epsilon n$. \label{stepyesdone}
\item If some $ b_k \ge \frac{2}{3} \epsilon^4 dn^2$, do the following:
\begin{enumerate}[a.]
\item Compute for each $l$
\[c_l=\sum_i a_{i,k}a_{i,l}
.\] \label{stepnocomputecs}
\item Let $T$ be either the set of $l$ such that $c_l>0$, or the set of $l$ such that $c_l<0$, whichever has a bigger sum in absolute value.\label{stepnofindT}
\item Compute for each $i \in [n]$ the values
\[d_T(i)=\sum_{k \in T} a_{i,k}
.\] \label{stepnocomputedT}
\item Let $S$ be either the set of $i \in [n]$ such that $d_{T}(i) > 0$ or the set of $i \in [n]$ such that
$d_T(i) < 0$, whichever has a bigger sum in absolute value. \label{stepnofindS}
\end{enumerate}
\end{enumerate}

Let us first analyze the running time. We can construct $M$ in time $(\log n)^{O(1)} dn$. We can compute each $s_{i,j}$ in $O(n)$ time, so computing all of them takes $O(dn^2)$ time in total. Computing each $b_k$ then similarly takes $O(dn)$ time (since we only need to sum the terms where $m_{i,j}>0$, and we keep a list of these entries), so that takes $O(dn^2)$ total time. If the algorithm says that $\|A\| \le \epsilon n$, then we are done. Otherwise, computing each $c_l$ can be done in time $O(n)$, so that takes $O(n^2)$ time in total. We then obtain $T$ in $O(n)$ time. Computing $S$ then similarly takes $O(n^2)$ time. Since $d=(C/\epsilon)^{O(1)}$, this shows that the algorithm runs in time $(C/\epsilon)^{O(1)}n^2$.

We now show that the algorithm is correct. First, we show the following lemma, which makes precise that we can use the expander to estimate the sum (\ref{tracefourthpower}).
\begin{lemma}
For any $k,l \in [n]$, we have
\begin{equation} \label{eqexpC4klclose}
\left|\sum_{i,j} a_{i,k}a_{i,l}a_{j,k}a_{j,l} -\frac{n}{d}\sum_{i,j}m_{i,j}a_{i,k}a_{i,l}a_{j,k}a_{j,l}\right| 
\le \frac{C^2n^2}{d^c} \le \frac{\epsilon^4}{3}n^2.
\end{equation}
\end{lemma}
\begin{proof}
Let $\va_{k,l}$ be the vector with entries $(\va_{k,l})_i=a_{i,k}a_{i,l}$. 
Since each $|a_{i,j}| \le C$, we have that $\|\va_{k,l}\|_2^2 \le C^2n$. Therefore, by (\ref{eqexpandersumsquare}),
\[\left|\left( \sum_i a_{i,k}a_{i,l} \right)^2 -\frac{n}{d}\va_{k,l}^\top M \va_{k,l}\right| 
\le \frac{C^2n^2}{d^c}.\]
Clearly
\[\left( \sum_i a_{i,k}a_{i,l} \right)^2=\sum_{i,j} a_{i,k}a_{i,l}a_{j,k}a_{j,l}
,\] and by the definition of $M$ and $\va_{k,l}$, we have
\[\va_{k,l}^\top M \va_{k,l}=\sum_{i,j}m_{i,j}a_{i,k}a_{i,l}a_{j,k}a_{j,l}
.\]
\end{proof}
\begin{lemma}
If the algorithm returns that $\|A\|_{2} \le \epsilon n$ then it is correct. 
\end{lemma}
\begin{proof}


We have
\[\sum_{k,l} \sum_{i,j}m_{i,j}a_{i,k}a_{i,l}a_{j,k}a_{j,l} = \sum_{i,j}m_{i,j} \langle \va_i,\va_j \rangle^2 = \sum_k\sum_{i,j}m_{i,j}a_{i,k}a_{j,k} \langle \va_i,\va_j \rangle = \sum_k b_k \le \frac{2}{3}\epsilon^4dn^3
.\]
Summing (\ref{eqexpC4klclose}) over all pairs $k,l \in [n]$, we have
\[\left| \sum_{k,l} \sum_{i,j}a_{i,k}a_{i,l}a_{j,k}a_{j,l} - \frac{n}{d} \sum_{k,l} \sum_{i,j}m_{i,j} a_{i,k}a_{i,l}a_{j,k}a_{j,l} \right| \le \frac{\epsilon^4}{3}n^4
.\]
Therefore, 
\[\tr AA^\top AA^\top=\sum_{i,j,k,l}a_{i,k}a_{i,l}a_{j,k}a_{j,l} \le \frac{n}{d}\sum_{i,j}m_{i,j} \langle \va_i, \va_j \rangle^2 + \frac{\epsilon^4}{3}n^4 \le \epsilon^4n^4
.\] Since $\tr A A^\top  A A^\top $ is the sum of the fourth powers of the singular values, this implies that each singular value is at most $\epsilon n$.
\end{proof}

\begin{lemma}
If the algorithm returns $S$ and $T$, then
\[\left|\sum_{(i,l) \in S \times T}a_{i,l}\right| \ge \frac{\epsilon^8}{100}n^2
.\]
\end{lemma}
\begin{proof}
First, note that for the particular $k$ we obtain in the algorithm, we have
\[ \frac{2}{3}\epsilon^4 dn^2 \le b_k = \sum_{i,j}m_{i,j}a_{i,k}a_{j,k}b_{i,j}=\sum_{i,j,l}m_{i,j} a_{i,k}a_{j,k} a_{i,l}a_{j,l}
.\]
we claim that we have
\[\sum_{l,i,j \in [n]} a_{i,k}a_{j,k} a_{i,l}a_{j,l} \ge \frac{n}{d}\sum_{i,j,l}m_{i,j} a_{i,k}a_{j,k} a_{i,l}a_{j,l} -\frac{\epsilon^4}{3}n^3\ge \frac{\epsilon^4}{3} n^3
.\]
Indeed, for any fixed $l$, by (\ref{eqexpC4klclose}), we have
\[\left|\sum_{i,j \in [n]} a_{i,k}a_{j,k} a_{i,l}a_{j,l} - \frac{n}{d}\sum_{i,j}m_{i,j} a_{i,k}a_{j,k} a_{i,l}a_{j,l} \right|\le \frac{\epsilon^4}{3}n^2
,\] and we can add this up over all $l \in [n]$.
Let $\vu =(a_{i,k})_{i=1}^n$, 
and $\vv$ be the vector with coordinates
\[v_l=\sum_i a_{i,k}a_{i,l}
.\]
Then $\|\vv\|_\infty \le n$, and 
$\|\vu\|_2 \le \sqrt{n}$ and we have
\[ \vu^\top A\vv \ge \frac{\epsilon^4}{3} n^3
.\]
Note however, that 
\[|\vu^\top A \vv| \le \|\vu^\top A\|_1 \|\vv\|_\infty
.\]
Therefore, we obtain that 
\[\|\vu^\top A\|_1 \ge \frac{\epsilon^4}{3}n^2 
.\]
Since $T$ consists of either the positive or the negative coordinates of $\vu^\top A$, whichever one has larger sum in absolute value, this implies that the $T$ that we obtain in step \ref{stepnofindT} satisfies, with $\bm{1}_T$ the characteristic vector, 
\[\left|\vu^\top A\bm{1}_T\right| \ge \frac{\epsilon^4}{6}n^2
.\] Since $\|\vu\|_2 \le \sqrt{n}$, by the Cauchy-Schwarz inequality, this implies that 
\[\|A\bm{1}_T\|_2^2 \ge\frac{(\vu^\top A\bm{1}_T)^2}{\|\vu\|_2^2} \ge \frac{\epsilon^8}{36}n^3
.\]
Since each row $\va_i$ of $A$ has $\|\va_i\|_2 \le \sqrt{n}$, we also have that 
\[\|A \bm{1}_T\|_\infty \le \sqrt{n}\|\bm{1}_T\|_2 \le n
.\]
Therefore, 
\[\|A \bm{1}_T\|_1 \ge \frac{\|A\bm{1}_T\|_2^2}{\|A\bm{1}_T\|_\infty} \ge \frac{\epsilon^8}{36}n^2
.\]
This means that for the $S$ that we obtain in step \ref{stepnofindS}, we have, if $\bm{1}_S$ is the characteristic vector,
\[\left|\bm{1}_S^\top A\bm{1}_T\right| \ge \frac{\epsilon^8}{72} n^2 \ge \frac{\epsilon^8}{100} n^2
,\] which is what we wanted to show.
\end{proof}
We have seen that either output of the algorithm must be correct, so this completes the proof of Theorem \ref{thmstep}.
\end{proof}

Before proving Theorem \ref{thmcomplalg}, we need one more technical lemma.
\begin{lemma} \label{lemmamakeeachrowcolumnbig}
	There exists an $O(n^2)$ time algorithm which takes as input a matrix $A \in \RR^{n \times n}$ and subsets $S,T \subseteq [n]$ such that 
	\[\sum_{\iskt} a_{i,k} \ge \epsilon' n^2
	,\] and outputs sets $S',T' \subseteq [n]$ such that
	\[\sum_{i \in S',k \in T'} a_{i,k} \ge \frac{2}{3}\epsilon' n^2
	.\] Furthermore, for any $i \in S'$, 
	\[\sum_{k \in T'} a_{i,k} \ge \frac{\epsilon'}{6}n 
	,\] and for any $k \in T'$, 
	\[\sum_{i \in S'} a_{i,k} \ge \frac{\epsilon'}{6}n 
	.\]
\end{lemma}

\begin{proof}
	Here is the algorithm.
	\begin{enumerate}[1.]
		\item To start, set $S'=S$ and $T'=T$.
		\item For each $i \in S'$ and each $k \in T'$, store the sum of the corresponding row or column in the submatrix induced by $S' \times T'$. \label{stepstoresums}
		\item Check whether there is a row or column with sum less than $\frac{\epsilon'}{6}n$.
		\item If there is, delete it, and update the row or column sums by subtracting the corresponding element from each sum.
		\item Go back to step 3 and repeat until no such row or column remains. 
	\end{enumerate}
	We first show that the running time is $O(n^2)$. We can compute each row and column sum in $O(n)$ time, therefore step \ref{stepstoresums} takes $O(n^2)$ time total. Each time we delete an element from $S'$ or $T'$, we perform $O(n)$ subtractions. The loops runs for at most $2n$ iterations since $|S|+|T| \le 2n$. Thus the algorithm takes $O(n^2)$ time.
	
	We next show that the algorithm is correct. In each step, the sum decreases by at most $\frac{\epsilon'}{6}n$, and there are at most $2n$ steps total. Therefore, after this process, for the $S'$ and $T'$ that we kept, we must still have
	\[\sum_{(i,k) \in S \times T} a_{i,k} \ge \frac{2}{3}\epsilon' n^2
	.\] In particular, this implies that when the algorithm terminates, $S'$ and $T'$ cannot be empty. By the definition of the algorithm, if it terminates, we must have the property that for any $i \in S'$, 
	\[\sum_{k \in T'} a_{i,k} \ge \frac{\epsilon'}{6}n 
	,\] and for any $k \in T'$, 
	\[\sum_{i \in S'} a_{i,k} \ge \frac{\epsilon'}{6}n 
	.\]
	This completes the proof of the lemma.
\end{proof}
We are now ready to prove our main theorem.
\begin{proof}[Proof of Theorem \ref{thmcomplalg}]
Let $\epsilon'=\epsilon^8/100$. Here is the algorithm. We iteratively construct a sequence of matrices as follows.
\begin{enumerate}[1.]
	\item Set $A_0=A$.
	\item For each $l$ starting at $0$, do the following.
	\begin{enumerate}[a.]
		\item Apply the algorithm from Theorem \ref{thmstep} to $A_l$.
		\item If the algorithm returns that $\|A_l\| \le \epsilon n$, then FINISH.
		\item Otherwise, the algorithm outputs sets $S,T \subseteq [n]$ such that
		\[\left|\sum_{i \in S,k \in T} a_{i,k} \right| \ge \epsilon' n^2
		.\] Let $\sigma \in \{-1,1\}$ be the sign of the above sum.
		\item Use Lemma \ref{lemmamakeeachrowcolumnbig}, applied to $\sigma A_l$ (and $S$, $T$ from above), to find $S',T' \subseteq [n]$ such that
		\[\sigma\sum_{i \in S',k \in T'} a_{i,k} \ge \frac{2}{3}\epsilon' n^2
		.\] Furthermore, for any $i \in S'$, 
		\[\sigma\sum_{k \in T'} a_{i,k} \ge \frac{\epsilon'}{6}n 
		,\] and for any $k \in T'$, 
		\[\sigma\sum_{i \in S'} a_{i,k} \ge \frac{\epsilon'}{6}n 
		.\]
		Replace $S$ and $T$ with $S'$ and $T'$.
		\item Let $S_l=S$, $T_l=T$, $t=\sigma \frac{\epsilon'}{3}$, and 
		$A_{l+1}=A_l - tK_{S_l,T_l}$.
	\end{enumerate}
\end{enumerate}

Let us first show that we can indeed apply Theorem \ref{thmstep} to each $A_l$. We first show that if $\vv$ is a row or column of $A_l$, then
\[\|\vv\|_2^2 \le n
.\] By the assumptions of the theorem, this is true for $l=0$. Fix $l$ so that it is true for $A_l$, let $a_{i,k}$ be the entries of $A_l$, and let $i$ be any row. If $i \notin S$, then the row does not change, so the $L^2$ norm of the row does not change in $A_{l+1}$. If $i \in S$, then we have 
\[\sum_{k \in T} a_{i,k}^2 - \left(a_{i,k}-t \right)^2 = 2t\sum_{k \in T} a_{i,k} - |T|t^2 \ge 2t\sigma\frac{\epsilon'}{6}n - t^2n= t\left(\sigma\frac{\epsilon'}{3}-t\right)n = 0
.\] Since the entries in $A_l$ were $a_{i,k}$, and the entries in $A_{l+1}$ are $a_{i,k}-t$, this implies that the $L^2$-norm of the corresponding row in $A_{l+1}$ cannot increase, and so for each row it is still at most $\sqrt{n}$. The analogous argument for columns shows that the same holds for each column.

Next, note that each entry of $A_0$ has absolute value at most $1$, and each entry changes by at most $\frac{\epsilon'}{3}$ when going from $A_l$ to $A_{l+1}$. Therefore, each entry of $A_l$ is at most $1+l\epsilon'/3$ in absolute value so we can apply Theorem \ref{thmstep} with $C=1+l\epsilon'/3$.

Finally, we show that the Frobenius norms of the matrices must decrease:
\[\|A_{l+1}\|_F^2 \le \|A_{l}\|_F^2- \frac{\epsilon'^2}{3}n^2
.\]
Let $a_{i,k}$ be the entries of $A_l$ again. We have 
\begin{multline*}
\|A_l\|_F^2-\|A_{l+1}\|_F^2=\sum_{\iskt} a_{i,k}^2-(a_{i,k}-t)^2= 2t\sum_{\iskt} a_{i,k} - |S||T|t^2 \\ \ge \sigma \frac{4}{3}t\epsilon' n^2 - |S||T|t^2 \ge \left(\sigma\frac{4}{3}t\epsilon'-t^2 \right)n^2
.\end{multline*} With our choice of $t=\sigma\frac{\epsilon'}{3}$, this is implies that
\[\|A_{l+1}\|_F^2 \le \|A_{l}\|_F^2- \frac{\epsilon'^2}{3}n^2
.\]

Now, we must have $\|A_0||_F^2 \le n^2$. Since the square of the Frobenius norm decreases by at least $\frac{\epsilon'^2}{3}n^2=\frac{\epsilon^{16}}{30000}n^2$ in each step, the number of steps is at most $O(1/\epsilon^{16})$. Therefore, after at most $O(1/\epsilon^{16})$, the algorithm must terminate.

As for the running time, the algorithm from Theorem \ref{thmstep} (with $C=1+l\epsilon'/3$) takes at most $O((l/\epsilon)^{O(1)}n^2)=\epsilon^{-O(1)}n^2$ time as $l=O(\epsilon^{-16})$. The algorithm from Lemma \ref{lemmamakeeachrowcolumnbig} takes $O(n^2)$ time. Finally, as the number of steps is $O(\epsilon^{-16})$, the whole process takes $\epsilon^{-O(1)}n^2$ time.
\end{proof}

\section{Approximation algorithm for subgraph counts}
\label{sec:counts}

We would like to approximate the number of copies of a fixed $k$-vertex graph $H$ in an $n$-vertex graph $G$, up to an additive error of at most $\epsilon n^k$. In this section, we prove Theorem~\ref{thm:countingalgthmgood}, which claims an algorithm to perform the task in $O(\epsilon^{-O_H(1)} n^2)$ time.

%
%
%

It will be cleaner to work instead with $\hom(H, G)$, the number of graph homomorphisms from $H$ to $G$. This quantity differs from the number of (labeled) copies of $H$ in $G$ by a negligible $O_H(n^{v(H)-1})$ additive error. We use the following multipartite version.

\begin{definition}
Let $H$ be a graph on $[k]$, and let $G$ be a $k$-partite weighted graph with vertex sets $V_1, \dots, V_k$. We write
\begin{equation}\label{eq:hom-star}
\hom^*(H,G)=\sum_{(v_1, \dots, v_k) \in V_1\times \cdots \times V_k} \prod_{\{i,j\} \in E(H)} G(v_i,v_j),
\end{equation}
where $G(x,y)$ denotes the edge-weight of $\{x,y\}$ in $G$, as usual.
\end{definition}

Note that for graphs $H$ and $G$, $\hom^*(H, G)$ counts the number graph homomorphisms from $H$ to $G$ where every vertex $v_i \in V(H)$ is mapped to the associated vertex part $V_i$ in $G$.

For every graph $G$, there is a $k$-partite $G^*$, obtained by replicating each vertex of $G$ into $k$ identical copies and two vertices of $G^*$ are adjcent if the original vertices in $G$ they came from are adjacent, such that $\hom(H, G) = \hom^*(H, G^*)$. Thus Theorem~\ref{thm:countingalgthmgood} follows from its multipartite generalization below.

\begin{theorem}
There exists a deterministic algorithm that takes as input a graph $H$ on $[k]$, a $k$-partite graph $G$ with each vertex part having at most $n$ vertices, and $\epsilon>0$, and outputs, in time $\epsilon^{-O_H(1)} n^2$, a quantity that approximates $\hom^*(H,G)$ up to an additive error of at most $\epsilon n^k$.
\end{theorem}

\begin{proof}
We begin with a description of the algorithm. If $H$ has no edges, then $\hom^*(H, G)= |V_1| \cdots |V_k|$. Assume now that $H$ has at least one edge, say $\{1,2\}$ (relabeling if necessary). Denote the vertex parts of $G$ by $V_1, \dots, V_k$.  Let $G_{12}$ denote the bipartite graph induced by $V_1$ and $V_2$ in $G$, and $d(G_{12})=d(V_1,V_2)=e(V_1,V_2)/(|V_1||V_2|)$ to denote the edge density between $V_1$ and $V_2$ in $G$. By Theorem \ref{thmalgweakregbest}, we can algorithmically find $S_1, \dots, S_r \subseteq V_1$, $T_1, \dots, T_r \subseteq V_2$, and $c_1, \dots, c_r = O(\epsilon^8)$, with $r = O(\epsilon^{-16})$, such that the weighted bipartite graph $G_{12}'$ on vertex sets $V_1$ and $V_2$ defined by
\begin{equation}\label{eq:G'12}
G_{12}'= d(G_{12}) + \sum_{i=1}^r c_i K_{S_i,T_i}
\end{equation}
satisfies
\[
d_\square(G_{12},G'_{12}) \le \epsilon/2.
\]
Let $G^{(i)}$ be $G$ obtained by deleting the vertices $(V_1\setminus S_i) \cup (V_2 \setminus T_i)$. Let $H'$ be $H$ with edge $\{1,2\}$ removed. 
Since $H'$ has one fewer edge than $H$, we can recursively apply the algorithm to estimate each of $\hom^*(H', G)$, $\hom^*(H', G^{(1)}), \dots, \hom^*(H', G^{(r)})$ up to an additive error of at most $c\epsilon^9$, where $c$ is some absolute constant. Summing up a linear combination of these estimates, we obtain an estimate for
\[
d(G_1,G_2)\hom^*(H',G) +\sum_{i=1}^r c_i \hom^*(H',G^{(i)}),
\]
which we use as our estimate for $\hom^*(H, G)$.

Now we prove the correctness of the algorithm. Let $G'$ be obtained from $G$ by replacing the bipartite graph between $V_1$ and $V_2$ by $G'_{12}$. We claim that
\begin{equation} \label{eq:counting-lem-proof}
\abs{\hom^*(H, G) - \hom^*(H, G')} \le \frac{\epsilon n^k}{2}.
\end{equation}
Indeed,
\[
\hom^*(H, G) - \hom^*(H, G')
= \sum_{(v_1, \dots, v_k) \in V_1\times \cdots \times V_k}f_{v_3, \dots, v_k}(v_1)g_{v_3, \dots, v_k}(v_2) (G(v_1, v_2) - G'(v_1, v_2))
\]
for some $f_{v_3, \dots, v_k}(v_1),g_{v_3, \dots, v_k}(v_2) \in \{0,1\}$ obtained by appropriately grouping the $G(v_i, v_j)$ factors in \eqref{eq:hom-star}.\footnote{We use the assumption that $0 \le G \le 1$ in this step. In the analogous step in \cite{FLZ17}, we mistakenly also assumed that $0 \le G' \le 1$, which is not necessarily the case.} 
For fixed $(v_3, \dots, v_k) \in V_3 \times \cdots \times V_k$, we have
\begin{multline*}
\abs{\sum_{(v_1, v_2) \in V_1\times V_2} f_{v_3, \dots, v_k}(v_1)g_{v_3, \dots, v_k}(v_2) (G(v_1, v_2) - G'(v_1, v_2))}
\\
\le \max_{U \subset V_1, W\subset V_2} |e_{G_{12}}(U,W) - e_{G'_{12}}(U,W)| \le n^2 d_\square(G_{12},G'_{12}) \le \epsilon n^2/2.
\end{multline*}
Then, summing over all $(v_3, \dots, v_k) \in V_3\times \cdots \times V_k$ and applying the triangle inequality, we obtain \eqref{eq:counting-lem-proof}.

From \eqref{eq:G'12}, we have
\[
\hom^*(H,G') = d(G_{12})\hom^*(H',G) +\sum_{i=1}^r c_i \hom^*(H',G^{(i)}).
\]
Since $c_i = O(\epsilon^8)$ and $r = O(\epsilon^{-16})$, we obtain an estimate of $\hom^*(H,G')$ up to an additive error of at most $\epsilon n^k/2$ as long as each $\hom^*(H', -)$ in the above sum is estimated up to an additive error $c\epsilon^9$ for an appropriate positive constant $c$. Together with \eqref{eq:counting-lem-proof}, the estimate is within $\epsilon n^k$ of $\hom^*(H, G)$, as claimed.

Now we analyze the running time. It takes $\epsilon^{-O(1)} n^2$ time (independent of $H$) to find $S_1, \dotsc, S_r$, $T_1, \dotsc, T_r$, and $c_1, \dotsc, c_r$. Estimating each $\hom^*(H', G)$, $\hom^*(H', G^{(1)}), \dots, \hom^*(H', G^{(r)})$ up to an additive error of at most $c\epsilon^9$ takes $\epsilon^{-O_{H'}(1)} n^2$ time (by induction), and we need to perform $r+1 = O(\epsilon^{-16})$ such estimates. Thus the total running time is $\epsilon^{-O_H(1)}n^2$.
\end{proof}
%

\bibliography{ref}
\bibliographystyle{amsplain_mod2}

\end{document}